\documentclass[11pt, reqno, a4paper]{amsart}

\usepackage[dvipsnames]{xcolor}
\usepackage[pagebackref,backref=true,colorlinks=true, linkcolor=Sepia, citecolor=Sepia]{hyperref}
\usepackage{amsmath,amsthm,amssymb,fancyhdr,graphicx,bbm,cancel,mathrsfs,todonotes,xypic,pinlabel}
\usepackage{tikz}
\usepackage{extarrows,tipa}
\usetikzlibrary{matrix}
\usepackage[all]{xy}
\usepackage{mathtools}
\usepackage{bm}
\usepackage{verbatim}
\usepackage{microtype}
\usepackage{subcaption}

\usepackage[utf8]{inputenc}
\usepackage[T2A]{fontenc}
\usepackage[russian,english]{babel}

\DeclareFontFamily{T1}{cbgreek}{}
\DeclareFontShape{T1}{cbgreek}{m}{n}{<-6>  grmn0500 <6-7> grmn0600 <7-8> grmn0700 <8-9> grmn0800 <9-10> grmn0900 <10-12> grmn1000 <12-17> grmn1200 <17-> grmn1728}{}
\DeclareSymbolFont{quadratics}{T1}{cbgreek}{m}{n}
\DeclareMathSymbol{\qoppa}{\mathord}{quadratics}{19}
\DeclareMathSymbol{\Qoppa}{\mathord}{quadratics}{21}

\usepackage{amsthm}
\newtheorem{theorem}{Theorem}
\newtheorem{lemma}{Lemma}
\newtheorem{proposition}{Proposition}
\newtheorem{corollary}{Corollary}
\newtheorem{definition}{Definition}

\newtheorem{remark}{Remark}
\def\b{\ensuremath\boldsymbol}

\usepackage{mathtools} %--> for \Vert (norm symbol)

\usepackage{enumerate,pdfpages,stmaryrd} 
\usepackage[margin=1.2in,footskip=.5in,marginparwidth=1in]{geometry}

\usepackage{tikz}
\usepackage{dsfont}
\usetikzlibrary{matrix}

\usepackage{textcase}

\begin{document}

\title{{\large\NoCaseChange{Similarity Algebra: A Framework for Approximate Algebraic \\and Lie Structures with Collapse to Classical Algebra}}}

\author{\NoCaseChange{Benyamin Ghojogh, Golbahar Amanpour\\
\\Waterloo, ON, Canada\\
The authors contributed equally to this work.\\
Correspondence goes to {\sc bghojogh@uwaterloo.ca}}}

\maketitle

\begin{abstract}
Classical algebraic structures require exact satisfaction of their defining axioms. We propose similarity algebra, a framework extending algebraic and Lie structures to settings where operations satisfy quantitative bounds up to a tolerance $\varepsilon$. Instead of strict associativity, inverses, or distributivity, we study families of operations controlled by explicit $\varepsilon$-estimates and analyze their behavior under limit collapse.

Under uniform error control and $C^1_{\mathrm{loc}}$ convergence of the structure maps, we prove a general collapse theorem showing that similarity structures converge to classical algebraic objects, as $\varepsilon \rightarrow 0$. We develop a hierarchy of approximate structures, including similarity groups, rings, fields, vector spaces, and Lie groups, formalized through axioms satisfied within metric distance $\varepsilon$. We further define a category of similarity algebras governing morphisms between approximate systems. Moreover, we clarify the relationship between similarity algebra and fuzzy algebra, showing that the former generalizes the latter. The proposed similarity algebra can be useful to model real-world phenomena where operations or relations are inherently approximate.
\end{abstract}


\section{Introduction}

Classical algebraic structures, such as groups, rings, fields, modules, vector spaces, and Lie algebras, are fundamental tools in mathematics. However, in many real-world applications, operations are often approximate rather than exact. For example, in machine learning, similarity measures or kernel functions frequently replace exact arithmetic \cite{scholkopf2018learning} (e.g., see \cite{amanpour2026wittgenstein}); in data science, aggregating noisy or approximate similarity scores requires flexible algebraic reasoning \cite{indyk1998approximate}; in network theory, connectivity and adjacency operations may only hold approximately \cite{barabasi2013network}; and in other areas of mathematics, such as functional analysis or metric-based algebraic systems, classical axioms can be too rigid to model practical phenomena \cite{hyers1941stability,rassias1978stability,gromov2007metric}.

To address this, we introduce \textit{Similarity Algebra}, a framework for approximate algebraic and Lie structures. This framework formalizes approximate associativity, inverses, distributivity, and Lie brackets, while providing collapse theorems that ensure convergence to classical structures as approximation errors vanish. By doing so, similarity algebra enables mathematically rigorous reasoning in settings where operations are inherently approximate, bridging the gap between classical algebraic theory and practical computation or modeling.

% The framework is particularly relevant in applications such as similarity-based machine learning, clustering, kernel methods, approximate computation in networks, and generalized algebraic modeling, and it lays the groundwork for exploring new theoretical and applied extensions in both mathematics and computational sciences.

% \subsection*{Related Work}

The formalization of approximate algebraic structures draws from several distinct mathematical traditions, most notably the stability theory of functional equations. The \textit{Hyers-Ulam-Rassias stability} framework, initiated by Hyers \cite{hyers1941stability} and significantly generalized by Rassias \cite{rassias1978stability}, examines whether a function that ``almost'' satisfies an algebraic identity—such as the Cauchy functional equation—is necessarily close to an exact solution. This perspective provides the historical and logical foundation for our $\varepsilon$-bounded axioms, suggesting that approximate algebraic behavior is not merely noise but a structured deviation from classical norms.

Furthering this metric-based approach, Gromov \cite{gromov2007metric} introduced profound insights into metric structures that transcend rigid classical algebra, demonstrating how algebraic and geometric properties can be preserved or transformed under limit conditions. In the context of operator theory, Pisier \cite{pisier1998similarity} explored the \textit{similarity degree} of operator algebras, quantifying the distance between bounded homomorphisms and more rigid *-homomorphisms. This line of inquiry is complemented by the work of Arora et al. \cite{adali1998multi}, which addresses the computational complexity of finding approximate solutions to algebraic problems, highlighting the practical necessity of a formal similarity framework.

Our work also intersects with the modern study of additive combinatorics. Tao \cite{tao2008product} formalized the notion of \textit{approximate groups}, which are sets that are ``almost'' closed under multiplication, where the product of elements stays within a controlled expansion of the original set. While approximate groups focus largely on set-theoretic growth, our similarity algebra extends these concepts to a full hierarchy of structures---including rings, modules, and Lie algebras---unified by a metric convergence that ensures a rigorous collapse to classical algebra as the tolerance parameter vanishes.

It is also important to distinguish similarity algebra from the well-established field of \textit{fuzzy algebra} \cite{mordeson2001fuzzy,rosenfeld1971fuzzy}. While fuzzy algebra uses membership functions to define the degree to which an element belongs to a set or satisfies a property, similarity algebra relies on a metric $d$ and a tolerance parameter $\varepsilon$ to measure the approximate satisfaction of axioms. A key difference lies in the behavior of these systems in the limit: while fuzzy structures represent a permanent relaxation of membership, similarity algebra is specifically designed with \textit{collapse theorems}. These theorems ensure that as the approximation error $\varepsilon$ vanishes, the similarity structures, ranging from groups to Lie algebras, mathematically converge to their exact classical counterparts, recovering the rigid axioms of standard algebra.

The remainder of this paper is organized as follows. Section \ref{section_similarity_algebra} introduces the fundamental definition of a \textit{similarity algebra} and develops a hierarchy of approximate structures, including similarity semigroups, monoids, groups, rings, fields, and vector spaces. Section \ref{section_similarity_lie_algebra} extends this framework to similarity Lie algebras and similarity Lie groups, defining the approximate Lie bracket and smooth manifold requirements. Section \ref{section_collapse} provides the core theoretical contribution of this work, i.e., the collapse theorems and the limit principle for approximate equality. In Section \ref{section_practical_examples}, we provide two practical examples for similarity algebra; one for similarity field and one for similarity perturbed matrix Lie groups. Section \ref{section_category} introduces the \textit{category of similarity algebra}, exploring the morphisms and structural mappings between approximate systems. Section \ref{section_fuzzy} offers a rigorous comparison between the proposed similarity framework and fuzzy algebra, highlighting the unique convergence properties of our model. Finally, Section \ref{section_conclusion} concludes the paper with a discussion on the implications of this framework for machine learning and computational modeling.

\section{Similarity Algebra and Its Structures}\label{section_similarity_algebra}

In this section, we establish the formal foundation of similarity algebra by introducing a hierarchy of approximate algebraic structures. We begin with the most basic building blocks---similarity semigroups and monoids---and progressively construct more complex systems, including similarity groups, rings, fields, and vector spaces. Unlike classical algebra, which assumes absolute precision in its operations, our framework acknowledges a tolerance $\varepsilon$ within a metric space $(A, d)$. By redefining core axioms such as associativity, identity, and distributivity as $\varepsilon$-bounded constraints, we create a flexible yet mathematically rigorous environment. This hierarchical approach allows us to generalize the concept of ``approximate equality'' across diverse algebraic domains, ensuring that each structure maintains a well-defined relationship with its classical counterpart through a shared metric foundation.

\begin{definition}[Similarity Algebra]
A similarity algebra is a tuple:
\begin{align}
\mathcal{A}_\varepsilon := (A, \{\circ_i^\varepsilon\}_{i \in I}, d, \varepsilon), 
\end{align}
where $A$ is a nonempty set, each $\circ_i^\varepsilon: A^{n_i} \rightarrow A$ is an $n_i$-ary operation, $I$ is the index set of operators, $d$ is a metric on $A$, and $\varepsilon$ is a small non-negative tolerance (approximation error):
\begin{align}
\varepsilon \gtrsim 0,
\end{align}
and each operation is uniformly Lipschitz: 
\begin{align}
d(\circ_i^\varepsilon(\b{x}), \circ_i^\varepsilon(\b{y})) \leq L_i \sum_j d(x_j, y_j), \quad \forall x_j, y_j \in A,
\end{align}
in which $\b{x}$ and $\b{y}$ denote two vectors or sets with $x_j$ and $y_j$ being their $j$-th element, respectively, and $L_i$ is the Lipschitz constant for the operation $\circ_i^\varepsilon$.
\end{definition}

\begin{remark}
The metric or similarity function $d$ controls the “approximate satisfaction” of axioms in classical algebra. Assume $f_\varepsilon(x)$ and $g_\varepsilon(x)$ are approximate functions of $x$ with tolerance $\varepsilon$. Then:
\begin{align}\label{equation_d_f_g}
f_\varepsilon(x) \approx_\varepsilon g_\varepsilon(x) \quad \iff \quad d\big(f_\varepsilon(x), g_\varepsilon(x)\big) \leq \varepsilon.
\end{align}
The distance vanishes to zero when $\varepsilon \rightarrow 0$:
\begin{align}
\lim_{\varepsilon \rightarrow 0} d\big(f_\varepsilon(x), g_\varepsilon(x)\big) = 0 \quad\implies\quad f(x) = g(x).
\end{align}
As a special case of Eq. (\ref{equation_d_f_g}), we have:
\begin{align}\label{equation_d_x_y}
x \approx_\varepsilon y \quad \iff \quad d\big(x, y\big) \leq \varepsilon, \quad \forall x,y \in A.
\end{align}
\end{remark}

% As will be discussed in Section \ref{section_collapse}, the classical algebraic laws are satisfied up to similarity with tolerance $\epsilon$:
% \begin{align}
% \text{similarity algebraic law}(x_1, x_2, \dots, x_n) \approx_\varepsilon \text{classical algebraic law}(x_1, x_2, \dots, x_n).
% \end{align}
% This will be clarified more rigorously later in this paper.

\begin{definition}[Similarity Semigroup]
A similarity semigroup is a tuple:
\begin{align}
\mathcal{G}_\varepsilon := (A, *_\varepsilon, d, \varepsilon), 
\end{align}
where $*_\varepsilon: A \times A \rightarrow A$ is an approximate operator:
\begin{align}
\lim_{\varepsilon \rightarrow 0} *_\varepsilon = *,
\end{align}
and:
\begin{itemize}
\item there is approximate closure:
\begin{align}
x *_\varepsilon y \in_{\varepsilon} A, \quad \forall x, y \in A,
\end{align}
meaning that:
\begin{align}
\forall x, y \in A, \quad \exists z \in A, \quad \text{such that} \quad d(z, x *_\varepsilon y) \in A, 
\end{align}
\item there is approximate associativity:
\begin{align}\label{equation_approximate_associativity_multiplication}
d\big( (x *_\varepsilon y) *_\varepsilon z,\, x *_\varepsilon (y *_\varepsilon z) \big) \leq \varepsilon, \quad \forall x,y,z \in A,
\end{align}
\item there is stability:
\begin{align}
d(x *_\varepsilon y, x' *_\varepsilon y') \leq L\big(d(x, x') + d(y, y')\big), \quad \forall x,y,x',y' \in A,
\end{align}
in which $L$ is the Lipschitz constant. 
\end{itemize}
\end{definition}

\begin{definition}[Similarity Monoid]
A similarity monoid is a tuple:
\begin{align}
\mathcal{M}_\varepsilon := (A, *_\varepsilon, e_\varepsilon, d, \varepsilon),
\end{align}
which is a similarity semigroup together with an approximate identity element $e_\varepsilon \in A$ such that:
\begin{align}
\lim_{\varepsilon \rightarrow 0} e_\varepsilon = e,
\end{align}
and there exists approximate identity such that:
\begin{align}
& d(x *_\varepsilon e_\varepsilon, x) \leq \varepsilon, \quad \forall x \in A, \label{equation_approximate_identity_1} \\
& d(e_\varepsilon *_\varepsilon x, x) \leq \varepsilon, \quad \forall x \in A. \label{equation_approximate_identity_2}
\end{align}
\end{definition}

\begin{definition}[Similarity Abelian/Commutative Monoid]
A similarity Abelian (or commutative) monoid is a similarity monoid $(A, *_\varepsilon, e_\varepsilon, d, \varepsilon)$ such that:
\begin{itemize}
\item there is approximate commutativity:
\begin{align}\label{equation_approximate_commutativity_multplication}
d\big( (x *_\varepsilon y),\, (y *_\varepsilon x) \big) \leq \varepsilon, \quad \forall x,y \in A.
\end{align}
\end{itemize}
\end{definition}

\begin{definition}[Similarity Group]
A similarity group is a tuple:
\begin{align}
\mathcal{G}_\varepsilon := (A, *_\varepsilon, e_\varepsilon, (\cdot)_\varepsilon^{-1}, d, \varepsilon),
\end{align}
which is a similarity monoid satisfying approximate associativity (\ref{equation_approximate_associativity_multiplication}), approximate identity (\ref{equation_approximate_identity_1}), (\ref{equation_approximate_identity_2}), and additional features:
\begin{itemize}
\item each $x \in A$ has an approximate inverse $x_\varepsilon^{-1} \in A$ satisfying:
\begin{align}
& \lim_{\varepsilon \rightarrow 0} (x_\varepsilon)^{-1} = x^{-1}, \\
& d(x *_\varepsilon x_\varepsilon^{-1}, e_\varepsilon) \leq \varepsilon, \label{equation_approximate_multiplicative_inverse_1} \\
& d(x_\varepsilon^{-1} *_\varepsilon x, e_\varepsilon) \leq \varepsilon, \label{equation_approximate_multiplicative_inverse_2}
\end{align} 
\item and there is inverse stability:
\begin{align}
& d(x_\varepsilon^{-1}, y_\varepsilon^{-1}) \leq C d(x, y),
\end{align}
where $C$ is the Lipschitz constant. 
\end{itemize}
\end{definition}

\begin{definition}[Similarity Abelian/Commutative Group]
A similarity Abelian (or commutative) group is a similarity group, which also satisfies approximate commutativity of multiplication (\ref{equation_approximate_commutativity_multplication}).
\end{definition}

\begin{definition}[Similarity Semiring]
A similarity semiring is a tuple:
\begin{align}
\mathcal{R}_\varepsilon := (A, +_\varepsilon, *_\varepsilon, d, \varepsilon),
\end{align}
with two operators, approximate addition $+_\varepsilon$ and approximate multiplication $*_\varepsilon$, such that $(A, +_\varepsilon)$ is a similarity Abelian monoid under approximate addition, having:
\begin{itemize}
\item approximate commutativity:
\begin{align}\label{equation_semiring_approximate_commutativity_addition}
d\big( (x +_\varepsilon y),\, (y +_\varepsilon x) \big) \leq \varepsilon, \quad \forall x,y \in A,
\end{align}
\item approximate associativity:
\begin{align}\label{equation_semiring_approximate_associativity_addition}
d\big( (x +_\varepsilon y) +_\varepsilon z,\, x +_\varepsilon (y +_\varepsilon z) \big) \leq \varepsilon, \quad \forall x,y,z \in A,
\end{align}
\item there exists approximate addition identity $0_\varepsilon \in A$ such that:
\begin{align}
& d(x +_\varepsilon 0_\varepsilon, x) \leq \varepsilon, \quad \forall x \in A, \label{equation_semiring_addition_identity_1} \\
& d(0_\varepsilon +_\varepsilon x, x) \leq \varepsilon, \quad \forall x \in A, \label{equation_semiring_addition_identity_2}
\end{align}
\item there exists approximate commutativity of addition:
\begin{align}\label{equation_approximate_commutativity_addition}
d\big( (x +_\varepsilon y),\, (y +_\varepsilon x) \big) \leq \varepsilon, \quad \forall x,y \in A.
\end{align}
\end{itemize}
and $(A, *_\varepsilon)$ is a similarity monoid under approximate multiplication, having:
\begin{itemize}
\item approximate associativity:
\begin{align}\label{equation_semiring_approximate_associativity_multiplication}
d\big( (x *_\varepsilon y) *_\varepsilon z,\, x *_\varepsilon (y *_\varepsilon z) \big) \leq \varepsilon, \quad \forall x,y,z \in A,
\end{align}
\item there exists approximate multiplication identity $e_\varepsilon$ such that:
\begin{align}
& d(x *_\varepsilon e_\varepsilon, x) \leq \varepsilon, \quad \forall x \in A, \label{equation_semiring_multiplication_identity_1} \\
& d(e_\varepsilon *_\varepsilon x, x) \leq \varepsilon, \quad \forall x \in A. \label{equation_semiring_multiplication_identity_2}
\end{align}
\end{itemize}
and there is approximate distributivity of approximate multiplication over approximate addition:
\begin{align}\label{equation_semiring_distributivity}
d\big( x *_\varepsilon (y +_\varepsilon z),\, (x *_\varepsilon y) +_\varepsilon (x *_\varepsilon z) \big) \leq \varepsilon, \quad \forall x,y,z \in A.
\end{align}
\end{definition}

\begin{definition}[Similarity Ring]
A similarity ring is a tuple:
\begin{align}
\mathcal{R}_\varepsilon := (A, +_\varepsilon, *_\varepsilon, -_\varepsilon, 0_\varepsilon, d, \varepsilon),
\end{align}
which is a similarity semiring where:
\begin{itemize}
\item $(A, +_\varepsilon)$ is a similarity Abelian group under approximate addition, satisfying Eqs. (\ref{equation_semiring_approximate_commutativity_addition}), (\ref{equation_semiring_approximate_associativity_addition}), (\ref{equation_semiring_addition_identity_1}), (\ref{equation_semiring_addition_identity_2}), and (\ref{equation_approximate_commutativity_addition}), and there is approximate additive inverse $(-x)_\varepsilon \in A$ for each $x \in A$:
\begin{align}
&d\big(x +_\varepsilon (-x)_\varepsilon, 0_\varepsilon \big) \leq \varepsilon, \label{equation_approximate_additive_inverse_1} \\
&d\big((-x)_\varepsilon +_\varepsilon x, 0_\varepsilon \big) \leq \varepsilon, \label{equation_approximate_additive_inverse_2}
\end{align}
where $0_\varepsilon$ is the approximate zero such that:
\begin{align}
&\lim_{\varepsilon \rightarrow 0} 0_\varepsilon = 0, \\
&\lim_{\varepsilon \rightarrow 0} (-x)_\varepsilon = -x.
\end{align}
\item $(A, *_\varepsilon)$ is a similarity monoid under approximate multiplication, satisfying Eqs. (\ref{equation_semiring_approximate_associativity_multiplication}), (\ref{equation_semiring_multiplication_identity_1}), and (\ref{equation_semiring_multiplication_identity_2}). 
\item there is approximate distributivity of approximate multiplication over approximate addition, as in Eq. (\ref{equation_semiring_distributivity}).
\end{itemize}
\end{definition}

\begin{definition}[Similarity Field]
A similarity field is a tuple:
\begin{align}
\mathcal{F}_\varepsilon := (A, +_\varepsilon, *_\varepsilon, -_\varepsilon, 0_\varepsilon, e_\varepsilon, (\cdot)_\varepsilon^{-1}, d, \varepsilon),
\end{align}
which is a similarity ring where:
\begin{itemize}
\item $(A, +_\varepsilon)$ is a similarity Abelian group, satisfying Eqs. (\ref{equation_semiring_approximate_commutativity_addition}), (\ref{equation_semiring_approximate_associativity_addition}), (\ref{equation_semiring_addition_identity_1}), (\ref{equation_semiring_addition_identity_2}), (\ref{equation_approximate_additive_inverse_1}), and (\ref{equation_approximate_additive_inverse_2}).
\item $(A \backslash \{0_\varepsilon\}, *_\varepsilon)$ is a similarity Abelian group satisfying Eqs. (\ref{equation_semiring_approximate_associativity_multiplication}), (\ref{equation_semiring_multiplication_identity_1}), and (\ref{equation_semiring_multiplication_identity_2}), and there is approximate multiplicative inverse (\ref{equation_approximate_multiplicative_inverse_1}) and (\ref{equation_approximate_multiplicative_inverse_2}).
\item there is approximate distributivity of approximate multiplication over approximate addition, as in Eq. (\ref{equation_semiring_distributivity}).
\end{itemize}
\end{definition}

\begin{definition}[Similarity Module]
Let $\mathcal{R}_\varepsilon = (A, +_\varepsilon, *_\varepsilon, d, \varepsilon)$ be a similarity ring. A similarity $\mathcal{R}_\varepsilon$-module is a module over the similarity ring $\mathcal{R}_\varepsilon$, defined as a tuple:
\begin{align}
\mathcal{M}_\varepsilon := (M, +_\varepsilon, \boldsymbol{\cdot}_\varepsilon, d),
\end{align}
such that:
\begin{itemize}
\item $(M, +_\varepsilon)$ is a Similarity Abelian group, 
\item $\boldsymbol{\cdot}_\varepsilon: A \times M \rightarrow M$,
\item there is approximate scalar distributivity (approximate module axioms):
\begin{align}
d\big( r \boldsymbol{\cdot}_\varepsilon (x +_\varepsilon y), (r \boldsymbol{\cdot}_\varepsilon x) +_\varepsilon (r \boldsymbol{\cdot}_\varepsilon y) \big) \leq \varepsilon,
\end{align}
where $r$ is a scalar. 
\end{itemize}
\end{definition}

\begin{definition}[Similarity Vector Space]
A similarity vector space $\mathcal{V}_\varepsilon = (V, +_\varepsilon, \boldsymbol{\cdot}_\varepsilon, d)$ is a similarity module over a similarity field $\mathcal{F}_\varepsilon$.
\end{definition}

\section{Similarity Lie Algebra and Similarity Lie Group}\label{section_similarity_lie_algebra}

Building upon the discrete similarity structures established in the previous section, we now extend the framework to the continuous domain by defining similarity Lie groups and similarity Lie algebras. While classical Lie theory relies on the rigid differential structure of smooth manifolds and the exact satisfaction of the Jacobi identity, many physical and computational systems exhibit symmetries that are only approximately preserved. In this section, we formalize these phenomena by introducing the similarity Lie bracket, which satisfies the antisymmetry and Jacobi identities within an $\varepsilon$ tolerance. We further explore the geometric implications of these structures, defining how the local properties of a similarity Lie group---governed by approximate group operations on a manifold---relate to its corresponding similarity Lie algebra. This extension provides the necessary mathematical language to treat approximate continuous symmetries with the same degree of foundational rigor applied to basic similarity groups and rings.

\begin{definition}[Similarity Lie Algebra]
A similarity Lie algebra is a vector space $\mathfrak{g}$ over a similarity field $\mathcal{F}_\varepsilon$ is a tuple:
\begin{align}
\mathcal{L}_\varepsilon := (\mathfrak{g}_\varepsilon, +_\varepsilon, [\cdot, \cdot]_\varepsilon, 0_\varepsilon, -_\varepsilon, d),
\end{align}
where $(\mathfrak{g}_\varepsilon, +_\varepsilon, 0_\varepsilon, -_\varepsilon)$ is a similarity vector space $\mathfrak{g}_\varepsilon$ over approximate addition $+_\varepsilon$, approximate additive inverse $-_\varepsilon$, and scalar multiplication, with the metric $d$. 
The $[\cdot, \cdot]_\varepsilon: \mathfrak{g}_\varepsilon \times \mathfrak{g}_\varepsilon \rightarrow \mathfrak{g}_\varepsilon$ is a similarity Lie bracket satisfying approximate axioms:
\begin{itemize}
\item Approximate bilinearity:
\begin{align}
& d([ax +_\varepsilon by, z]_\varepsilon, a[x, z]_\varepsilon + b[y, z]_\varepsilon) \leq \varepsilon, \quad \forall x,y,z \in \mathfrak{g}_\varepsilon, \forall a,b \in \mathbb{R}, \\
& d([z, ax +_\varepsilon by]_\varepsilon, a[z, x]_\varepsilon + b[z, y]_\varepsilon) \leq \varepsilon, \quad \forall x,y,z \in \mathfrak{g}_\varepsilon, \forall a,b \in \mathbb{R},
\end{align}
\item Approximate antisymmetry:
\begin{align}
& d([x, y]_\varepsilon - [y, x]_\varepsilon) \leq \varepsilon, \quad \forall x,y \in \mathfrak{g}_\varepsilon,
\end{align}
\item Approximate Jacobi identity:
\begin{align}
& d\big( [x, [y, z]_\varepsilon]_\varepsilon +_\varepsilon [y, [z, x]_\varepsilon]_\varepsilon +_\varepsilon [z, [x, y]_\varepsilon]_\varepsilon, 0_\varepsilon \big) \leq \varepsilon, \quad \forall x,y,z \in \mathfrak{g}_\varepsilon.
\end{align}
\end{itemize}
\end{definition}

\begin{definition}[Similarity Lie Group]
A similarity Lie group is a tuple:
\begin{align}
\mathfrak{L}_\varepsilon := (G_\varepsilon, \boldsymbol{\cdot}_\varepsilon, e_\varepsilon, (\cdot)_\varepsilon^{-1}, d),
\end{align}
where $G_\varepsilon$ is a smooth manifold with metric $d$ and $\boldsymbol{\cdot}_\varepsilon: G_\varepsilon \times G_\varepsilon \rightarrow G_\varepsilon$ is approximate group multiplication satisfying:
\begin{itemize}
\item Approximate associativity:
\begin{align}
d((g \boldsymbol{\cdot}_\varepsilon h) \boldsymbol{\cdot}_\varepsilon k, g \boldsymbol{\cdot}_\varepsilon (h \boldsymbol{\cdot}_\varepsilon k)) \leq \varepsilon, \quad g,h,k \in G_\varepsilon, 
\end{align}
\item Approximate identity:
\begin{align}
&d((g \boldsymbol{\cdot}_\varepsilon e_\varepsilon), g) \leq \varepsilon, \quad \forall g \in G_\varepsilon, \\
&d(( e_\varepsilon \boldsymbol{\cdot}_\varepsilon g), g) \leq \varepsilon, \quad \forall g \in G_\varepsilon, 
\end{align}
\item Approximate inverse:
\begin{align}
&d((g \boldsymbol{\cdot}_\varepsilon g_\varepsilon^{-1}), e_\varepsilon) \leq \varepsilon, \quad \forall g \in G_\varepsilon, \\
&d(( g_\varepsilon^{-1} \boldsymbol{\cdot}_\varepsilon g), e_\varepsilon) \leq \varepsilon, \quad \forall g \in G_\varepsilon.
\end{align}
\end{itemize}
% and all operations are smooth in the approximate sense, i.e., charts satisfy Lipschitz-type bounds with error $\varepsilon$.
Additionally, for each $\varepsilon \ge 0$, $G_\varepsilon$ is assumed to be a $C^1$-smooth manifold. 
The multiplication map:
\begin{align}
m_\varepsilon : G_\varepsilon \times G_\varepsilon \to G_\varepsilon,
\end{align}
and inversion map:
\begin{align}
\iota_\varepsilon : G_\varepsilon \to G_\varepsilon,
\end{align}
are $C^1$-smooth.
We assume that the structure maps converge in the $C^1$ topology uniformly on compact subsets as $\varepsilon \to 0$, i.e.,
\begin{align}
m_\varepsilon \to m, \quad \iota_\varepsilon \to \iota,
\end{align}
in $C^1_{\mathrm{loc}}$ (locally continuously differentiable).
\end{definition}

\section{Collapse Theorems}\label{section_collapse}

% The classical algebraic laws are satisfied up to similarity with tolerance $\epsilon$:
% \begin{align}
% \text{similarity algebraic law}(x_1, x_2, \dots, x_n) \approx_\varepsilon \text{classical algebraic law}(x_1, x_2, \dots, x_n).
% \end{align}
% This will be clarified more rigorously later in this paper.

% All collapse theorems are instances of a single limit principle: Metric-stable algebraic identities become exact in the zero-tolerance limit.

% The collapse theorem guarantees that as $\varepsilon \rightarrow 0$, all approximate operations become exact, recovering classical Lie structures.

The foundational premise of similarity algebra is that classical algebraic identities are satisfied up to a similarity tolerance $\epsilon$. Formally, for any $n$-ary operation, the relationship between the similarity and classical laws can be expressed as:
\begin{align}
&\mathcal{L}_{\epsilon}(x_1, x_2, \dots, x_n) \approx_\varepsilon \mathcal{L}_{0}(x_1, x_2, \dots, x_n), \quad \text{or} \\
&d\left(\mathcal{L}_{\epsilon}(x_1, \dots, x_n), \mathcal{L}_{0}(x_1, \dots, x_n)\right) \leq \varepsilon,
\end{align}
where $\mathcal{L}_{\epsilon}$ denotes the similarity algebraic law and $\mathcal{L}_{0}$ denotes its classical counterpart. This approximation is governed by the metric $d$ such that the distance between the two expressions is bounded by $\varepsilon$. 

The various collapse theorems presented in the following are specific instances of a unified limit principle: metric-stable algebraic identities become exact in the zero-tolerance limit. This principle guarantees that as the approximation error $\varepsilon \rightarrow 0$, the similarity structures---including the complex manifold properties of similarity Lie groups---undergo a mathematical collapse to their classical counterparts. This process ensures the recovery of rigid classical structures, confirming that similarity algebra is a consistent extension of classical algebraic theory.

\begin{lemma}[Limit of Approximate Equality]\label{lemma_limit_approximate_equality}
Let $(X, d)$ be a metric space and $\varepsilon \geq 0$. All operations are uniformly Lipschitz and uniformly convergent. If for all $\varepsilon \geq 0$, 
\begin{align}\label{equation_d_f_g_epsilon}
d\big(f_\varepsilon(x), g_\varepsilon(x)\big) \leq \varepsilon \quad \text{or} \quad f_\varepsilon(x) \approx_\varepsilon g_\varepsilon(x),
\end{align}
and
\begin{align}\label{equation_f_g_converge}
f_\varepsilon \rightarrow f, \quad g_\varepsilon \rightarrow g, 
\end{align}
uniformly, as $\varepsilon \rightarrow 0$, then:
\begin{align}\label{equation_f_g_equal}
f(x) = g(x), \quad \forall x.
\end{align}
\end{lemma}
\begin{proof}
Using the triangle inequality twice, we have:
\begin{align*}
d\big(f(x), g(x)\big) &\leq d\big(f(x), f_\varepsilon(x)\big) + d\big(f_\varepsilon(x), g(x)\big) \\
&\leq d\big(f(x), f_\varepsilon(x)\big) + d\big(f_\varepsilon(x), g_\varepsilon(x)\big) + d\big(g_\varepsilon(x), g(x)\big)\\
&\overset{(\ref{equation_d_f_g_epsilon})}{\leq} \lVert f(x) - f_\varepsilon(x) \rVert_\infty + \varepsilon + \lVert g_\varepsilon(x) - g(x) \rVert_\infty.
\end{align*}
As $\varepsilon \rightarrow 0$, we have:
\begin{align*}
\lim_{\varepsilon \rightarrow 0} d\big(f(x), g(x)\big) &\overset{(\ref{equation_f_g_converge})}{\leq} \lVert f(x) - f(x) \rVert_\infty + 0 + \lVert g(x) - g(x) \rVert_\infty = 0.
\end{align*}
As $d\big(f(x), g(x)\big) \geq 0$ because of being a metric, thus $d\big(f(x), g(x)\big) = 0$, resulting in $f(x) = g(x)$, which is Eq. (\ref{equation_f_g_equal}). The limits and uniform convergence in metric spaces are justified according to \cite{bourbaki1989lie,rassias1978stability,hyers1941stability,forti1995hyers}.
\end{proof}

\begin{proposition}[Collapse of Similarity Algebra]
Let $\mathcal{A}_\varepsilon = (A, \{\circ_i^\varepsilon\}_{i \in I}, d, \varepsilon)$ be a similarity algebra such that:
\begin{itemize}
\item each $\circ_i^\varepsilon \rightarrow \circ_i$ uniformly on compact sets, and
\item all defining algebraic identities are satisfied up to error $\varepsilon$.
\end{itemize}
Then, as error goes to zero, i.e., $\varepsilon \rightarrow 0$, we have:
\begin{align}
\mathcal{A}_\varepsilon \rightarrow \mathcal{A}_0 = (A, \{\circ_i^\varepsilon\}_{i \in I}, d),
\end{align}
where $\mathcal{A}_0$ is a classical algebra satisfying all identities exactly. 
\end{proposition}
\begin{proof}
A similarity algebra consists of operations $\{\circ_i^\varepsilon\}_{i \in I}$ satisfying identities:
\begin{align*}
P(\circ_1^\varepsilon, \dots, \circ_n^\varepsilon)(x) \approx Q(\circ_1^\varepsilon, \dots, \circ_n^\varepsilon)(x).
\end{align*}
That is:
\begin{align*}
d\big(P_\varepsilon(x), Q_\varepsilon(x)\big) \leq \varepsilon.
\end{align*}
Assume uniform convergence:
\begin{align*}
\circ_i^\varepsilon \rightarrow \circ_i.
\end{align*}
Then by Lemma \ref{lemma_limit_approximate_equality}, as $\varepsilon \rightarrow 0$, we have:
\begin{align*}
P(x) = Q(x).
\end{align*}
Thus, all defining identities hold exactly in the limit.
\end{proof}

\begin{proposition}[Collapse of Similarity Semigroup]
Let $\mathcal{G}_\varepsilon = (A, *_\varepsilon, d, \varepsilon)$ be a similarity semigroup satisfying approximate associativity:
\begin{align*}
d\big( (x *_\varepsilon y) *_\varepsilon z,\, x *_\varepsilon (y *_\varepsilon z) \big) \leq \varepsilon, \quad \forall x,y,z \in A.
\end{align*}
If $*_\varepsilon \rightarrow *$ uniformly, then as $\varepsilon \rightarrow 0$, 
\begin{align}
(A, *_\varepsilon, d, \varepsilon) \rightarrow (A, *),
\end{align}
where $(A, *)$ is a classical semigroup.
\end{proposition}
\begin{proof}
We define:
\begin{align*}
&f_\varepsilon(x,y,z) := (x *_\varepsilon y) *_\varepsilon z, \\
&g_\varepsilon(x,y,z) := x *_\varepsilon (y *_\varepsilon z).
\end{align*}
According to the approximate associativity:
\begin{align*}
d\big( (x *_\varepsilon y) *_\varepsilon z,\, x *_\varepsilon (y *_\varepsilon z) \big) \leq \varepsilon,
\end{align*}
we have:
\begin{align*}
d(f_\varepsilon(x,y,z), g_\varepsilon(x,y,z)) \leq \varepsilon.
\end{align*}
As $\varepsilon \rightarrow 0$, we have $*_\varepsilon \rightarrow *$ uniformly.
Then by Lemma \ref{lemma_limit_approximate_equality}, as $\varepsilon \rightarrow 0$, we have:
\begin{align*}
(x * y) * z = x * (y * z),
\end{align*}
which is the exact associativity. Hence, $(A, *)$ is a classical semigroup.
\end{proof}

\begin{proposition}[Collapse of Similarity Monoid]
Let $\mathcal{M}_\varepsilon = (A, *_\varepsilon, e_\varepsilon, d, \varepsilon)$ be a similarity monoid, with an approximate identity element $e_\varepsilon$, such that:
\begin{align*}
& d(x *_\varepsilon e_\varepsilon, x) \leq \varepsilon, \quad \forall x \in A,  \\
& d(e_\varepsilon *_\varepsilon x, x) \leq \varepsilon, \quad \forall x \in A. 
\end{align*}
If $e_\varepsilon \rightarrow e$ and $*_\varepsilon \rightarrow *$ uniformly, then as $\varepsilon \rightarrow 0$, 
\begin{align}
(A, *_\varepsilon, e_\varepsilon, d, \varepsilon) \rightarrow (A, *, e),
\end{align}
where $(A, *, e)$ is a classical monoid with identity element $e$.
\end{proposition}
\begin{proof}
We define:
\begin{align*}
f_\varepsilon(x) = x *_\varepsilon e_\varepsilon.
\end{align*}
According to approximate identity:
\begin{align*}
d(x *_\varepsilon e_\varepsilon, x) \leq \varepsilon,
\end{align*}
we have:
\begin{align*}
d(f_\varepsilon(x), x) \leq \varepsilon.
\end{align*}
As $\varepsilon \rightarrow 0$, we have $e_\varepsilon \rightarrow e$ and $*_\varepsilon \rightarrow *$ uniformly.
Then by Lemma \ref{lemma_limit_approximate_equality}, as $\varepsilon \rightarrow 0$, we have:
\begin{align*}
x * e = x,
\end{align*}
which is the exact identity element. Likewise, it can be proved for $d(e_\varepsilon *_\varepsilon x, x) \leq \varepsilon$ collapsing into $e * x = x$. Hence, $(A, *, e)$ is a classical monoid.
\end{proof}

\begin{proposition}[Collapse of Similarity Group]
Let $\mathcal{G}_\varepsilon = (A, *_\varepsilon, e_\varepsilon, (\cdot)_\varepsilon^{-1}, d, \varepsilon)$ be a similarity group, with an approximate inverse $(\cdot)_\varepsilon^{-1}$, such that:
\begin{align*}
& d(x *_\varepsilon x_\varepsilon^{-1}, e_\varepsilon) \leq \varepsilon, \quad \forall x \in A, \\
& d(x_\varepsilon^{-1} *_\varepsilon x, e_\varepsilon) \leq \varepsilon, \quad \forall x \in A.
\end{align*} 
If $x *_\varepsilon x_\varepsilon^{-1} \rightarrow e$ and $x_\varepsilon^{-1} *_\varepsilon x \rightarrow e$ and $*_\varepsilon \rightarrow *$ uniformly, then as $\varepsilon \rightarrow 0$, 
\begin{align}
(A, *_\varepsilon, e_\varepsilon, (\cdot)_\varepsilon^{-1}, d, \varepsilon) \rightarrow (A, *, e, (\cdot)^{-1}),
\end{align}
where $(A, *, e, (\cdot)^{-1})$ is a classical group with inverse $(\cdot)^{-1}$.
\end{proposition}
\begin{proof}
We define:
\begin{align*}
f_\varepsilon(x) = x *_\varepsilon x_\varepsilon^{-1}, \quad g_\varepsilon = e_\varepsilon.
\end{align*}
According to approximate inverse:
\begin{align*}
d(x *_\varepsilon x_\varepsilon^{-1}, e_\varepsilon) \leq \varepsilon,
\end{align*}
we have:
\begin{align*}
d(f_\varepsilon(x), g_\varepsilon) \leq \varepsilon.
\end{align*}
As $\varepsilon \rightarrow 0$, we have $e_\varepsilon \rightarrow e$ and $*_\varepsilon \rightarrow *$ uniformly.
Then by Lemma \ref{lemma_limit_approximate_equality}, as $\varepsilon \rightarrow 0$, we have:
\begin{align*}
x * x^{-1} = e,
\end{align*}
which is the exact inverse. Likewise, it can be proved for $d(x_\varepsilon^{-1} *_\varepsilon x, e_\varepsilon) \leq \varepsilon$ collapsing into $x^{-1} * x = e$. Hence, $(A, *, e, (\cdot)^{-1})$ is a classical group.
\end{proof}

\begin{proposition}[Collapse of Similarity Semiring]
Let $\mathcal{R}_\varepsilon = (A, +_\varepsilon, *_\varepsilon, d, \varepsilon)$ be a similarity semiring, with approximate distributivity:
\begin{align*}
d\big( x *_\varepsilon (y +_\varepsilon z),\, (x *_\varepsilon y) +_\varepsilon (x *_\varepsilon z) \big) \leq \varepsilon, \quad \forall x,y,z \in A.
\end{align*}
If $+_\varepsilon \rightarrow +$ and $*_\varepsilon \rightarrow *$ uniformly, then as $\varepsilon \rightarrow 0$, 
\begin{align}
(A, +_\varepsilon, *_\varepsilon, d, \varepsilon) \rightarrow (A, +, *),
\end{align}
where $(A, +, *)$ is a classical semiring.
\end{proposition}
\begin{proof}
We define:
\begin{align*}
f_\varepsilon(x) = x *_\varepsilon (y +_\varepsilon z), \quad g_\varepsilon(x) = (x *_\varepsilon y) +_\varepsilon (x *_\varepsilon z).
\end{align*}
According to approximate distributivity:
\begin{align*}
d\big( x *_\varepsilon (y +_\varepsilon z),\, (x *_\varepsilon y) +_\varepsilon (x *_\varepsilon z) \big) \leq \varepsilon,
\end{align*}
we have:
\begin{align*}
d(f_\varepsilon(x), g_\varepsilon(x)) \leq \varepsilon.
\end{align*}
As $\varepsilon \rightarrow 0$, we have $*_\varepsilon \rightarrow *$ and $+_\varepsilon \rightarrow +$ uniformly.
Then by Lemma \ref{lemma_limit_approximate_equality}, as $\varepsilon \rightarrow 0$, we have:
\begin{align*}
x * (y + z) = (x * y) + (x * z),
\end{align*}
which is the exact distributivity. Additive and multiplicative associativity follow from previous propositions. Hence, $(A, +, *)$ is a classical semiring.
\end{proof}

\begin{proposition}[Collapse of Similarity Ring]
Let $\mathcal{R}_\varepsilon = (A, +_\varepsilon, *_\varepsilon, -_\varepsilon, 0_\varepsilon, d, \varepsilon)$ be a similarity ring.
If approximate additive inverses:
\begin{align}
&d\big(x +_\varepsilon (-x)_\varepsilon, 0_\varepsilon \big) \leq \varepsilon, \\
&d\big((-x)_\varepsilon +_\varepsilon x, 0_\varepsilon \big) \leq \varepsilon,
\end{align}
and approximate distributivity converge uniformly, then as $\varepsilon \rightarrow 0$, 
\begin{align}
(A, +_\varepsilon, *_\varepsilon, -_\varepsilon, 0_\varepsilon, d, \varepsilon) \rightarrow (A, +, *, -, 0),
\end{align}
where $(A, +, *, -, 0)$ is a classical ring.
\end{proposition}
\begin{proof}
We define:
\begin{align*}
f_\varepsilon(x) = x +_\varepsilon (-x)_\varepsilon, \quad g_\varepsilon(x) = 0_\varepsilon.
\end{align*}
According to approximate additive inverse:
\begin{align*}
d\big(x +_\varepsilon (-x)_\varepsilon, 0_\varepsilon \big) \leq \varepsilon,
\end{align*}
we have:
\begin{align*}
d(f_\varepsilon(x), g_\varepsilon(x)) \leq \varepsilon.
\end{align*}
As $\varepsilon \rightarrow 0$, we have $+_\varepsilon \rightarrow +$ and $(-x)_\varepsilon \rightarrow -x$ uniformly.
Then by Lemma \ref{lemma_limit_approximate_equality}, as $\varepsilon \rightarrow 0$, we have:
\begin{align*}
x + (-x) = 0,
\end{align*}
which is the exact additive inverse. 
Likewise, it can be proved for $d\big((-x)_\varepsilon +_\varepsilon x, 0_\varepsilon \big) \leq \varepsilon$ collapsing into $(-x) + x = 0$.
Distributivity also follows from previous propositions. Hence, $(A, +, *, -, 0)$ is a classical ring.
\end{proof}

\begin{proposition}[Collapse of Similarity Field]
Let $\mathcal{F}_\varepsilon = (A, +_\varepsilon, *_\varepsilon, -_\varepsilon, 0_\varepsilon, e_\varepsilon, (\cdot)_\varepsilon^{-1}, d, \varepsilon)$ be a similarity field.
If multiplicative inverses:
\begin{align*}
& d(x *_\varepsilon x_\varepsilon^{-1}, e_\varepsilon) \leq \varepsilon, \quad \forall x \in A, \\
& d(x_\varepsilon^{-1} *_\varepsilon x, e_\varepsilon) \leq \varepsilon, , \quad \forall x \in A.
\end{align*}
exist for all nonzero elements in the limit, then as $\varepsilon \rightarrow 0$, 
\begin{align}
(A, +_\varepsilon, *_\varepsilon, -_\varepsilon, 0_\varepsilon, e_\varepsilon, (\cdot)_\varepsilon^{-1}, d, \varepsilon) \rightarrow (A, +, *, -, 0, e, (\cdot)^{-1}),
\end{align}
where $(A, +, *, -, 0, e, (\cdot)^{-1})$ is a classical field.
\end{proposition}
\begin{proof}
We define:
\begin{align*}
f_\varepsilon(x) = x *_\varepsilon x_\varepsilon^{-1}, \quad g_\varepsilon = e_\varepsilon.
\end{align*}
According to approximate inverse:
\begin{align*}
d(x *_\varepsilon x_\varepsilon^{-1}, e_\varepsilon) \leq \varepsilon,
\end{align*}
we have:
\begin{align*}
d(f_\varepsilon(x), g_\varepsilon) \leq \varepsilon.
\end{align*}
As $\varepsilon \rightarrow 0$, we have $e_\varepsilon \rightarrow e$ and $*_\varepsilon \rightarrow *$ uniformly.
Then by Lemma \ref{lemma_limit_approximate_equality}, as $\varepsilon \rightarrow 0$, we have:
\begin{align*}
x * x^{-1} = e,
\end{align*}
which is the exact inverse. Likewise, it can be proved for $d(x_\varepsilon^{-1} *_\varepsilon x, e_\varepsilon) \leq \varepsilon$ collapsing into $x^{-1} * x = e$. 
Other characteristics of field also follow from previous propositions.
Hence, $(A, +, *, -, 0, e, (\cdot)^{-1})$ is a classical field.
\end{proof}

\begin{proposition}[Collapse of Similarity Module]
Let $\mathcal{M}_\varepsilon = (M, +_\varepsilon, \boldsymbol{\cdot}_\varepsilon, d)$ be a similarity module over a similarity ring $R_\varepsilon$.
If both $R_\varepsilon \rightarrow R$ and module actions converge uniformly, then as $\varepsilon \rightarrow 0$, 
\begin{align}
\mathcal{M}_\varepsilon \rightarrow \mathcal{M},
\end{align}
where $\mathcal{M}$ is a classical module over a ring $R$.
\end{proposition}
\begin{proof}
We define:
\begin{align*}
f_\varepsilon(x) = r \boldsymbol{\cdot}_\varepsilon (x +_\varepsilon y), \quad g_\varepsilon(x) = (r \boldsymbol{\cdot}_\varepsilon x) +_\varepsilon (r \boldsymbol{\cdot}_\varepsilon y).
\end{align*}
where $r$ is a scalar. 
According to approximate scalar distributivity:
\begin{align*}
d\big( r \boldsymbol{\cdot}_\varepsilon (x +_\varepsilon y), (r \boldsymbol{\cdot}_\varepsilon x) +_\varepsilon (r \boldsymbol{\cdot}_\varepsilon y) \big) \leq \varepsilon,
\end{align*}
we have:
\begin{align*}
d(f_\varepsilon(x), g_\varepsilon(x) ) \leq \varepsilon.
\end{align*}
As $\varepsilon \rightarrow 0$, we have $+_\varepsilon \rightarrow +$ and $\boldsymbol{\cdot}_\varepsilon \rightarrow \boldsymbol{\cdot}$ uniformly.
Then by Lemma \ref{lemma_limit_approximate_equality}, as $\varepsilon \rightarrow 0$, we have:
\begin{align*}
r \boldsymbol{\cdot} (x + y) = (r \boldsymbol{\cdot} x) + (r \boldsymbol{\cdot} y),
\end{align*}
which is the exact scalar distributivity (exact module axioms). 
Hence, $\mathcal{M}$ is a classical module.
\end{proof}

\begin{proposition}[Collapse of Similarity Vector Space]
Let $\mathcal{V}_\varepsilon = (V, +_\varepsilon, \boldsymbol{\cdot}_\varepsilon, d)$ be a similarity vector space over a similarity field $\mathcal{F}_\varepsilon$.
If scalar multiplication and vector addition converge uniformly, then as $\varepsilon \rightarrow 0$, 
\begin{align}
\mathcal{V}_\varepsilon \rightarrow \mathcal{V},
\end{align}
where $\mathcal{V}$ is a classical vector space. 
\end{proposition}
\begin{proof}
It is same as module, but with field scalars.
All vector space axioms follow by exact limit of approximate identities, according to Lemma \ref{lemma_limit_approximate_equality}.
\end{proof}

\begin{proposition}[Collapse of Similarity Lie Algebra]
Let $\mathcal{L}_\varepsilon = (\mathfrak{g}_\varepsilon, +_\varepsilon, [\cdot, \cdot]_\varepsilon, 0_\varepsilon, -_\varepsilon, d)$ be a similarity Lie algebra over a similarity field $\mathcal{F}_\varepsilon$, satisfying:
\begin{itemize}
\item approximate addition $+_\varepsilon$ with identity $0_\varepsilon$ and additive inverse $-_\varepsilon$,
\item and approximate Lie bracket $[\cdot, \cdot]_\varepsilon$ satisfying approximate bilinearity, antisymmetry, and Jacobi identity, 
\end{itemize}
and assume all operations converge uniformly:
\begin{align}
+_\varepsilon \rightarrow +, \quad -_\varepsilon \rightarrow -, \quad [\cdot, \cdot]_\varepsilon \rightarrow [\cdot, \cdot], \quad 0_\varepsilon \rightarrow 0.
\end{align}
Then, in the limit $\varepsilon \rightarrow 0$, we have:
\begin{align}
\lim_{\varepsilon \rightarrow 0} (\mathfrak{g}_\varepsilon, +_\varepsilon, [\cdot, \cdot]_\varepsilon, 0_\varepsilon, -_\varepsilon, d) = (\mathfrak{g}, +, [\cdot, \cdot], 0, -),
\end{align}
which is a classical Lie algebra over the field $\mathcal{F}$.
The approximate bilinearity collapses to exact bilinearity and approximate antisymmetry and Jacobi identity collapse to exact forms.
\end{proposition}
\begin{proof}
Step 1 --- Collapse of the Vector Space:
\begin{itemize}
\item $(\mathfrak{g}_\varepsilon, +_\varepsilon, 0_\varepsilon, -_\varepsilon)$ satisfies approximate addition:
\begin{align*}
x +_\varepsilon (y +_\varepsilon z) \approx_\varepsilon (x +_\varepsilon y) +_\varepsilon z, \quad x +_\varepsilon 0_\varepsilon \approx_\varepsilon x, \quad x +_\varepsilon (-x)_\varepsilon \approx_\varepsilon 0_\varepsilon.
\end{align*}
\item Uniform convergence: By assumption, we have $\forall x,y,z\in \mathfrak{g}_\varepsilon$:
\begin{align*}
\lim_{\varepsilon \rightarrow 0} (x +_\varepsilon y) = x + y, \quad \lim_{\varepsilon \rightarrow 0} 0_\varepsilon = 0, \quad \lim_{\varepsilon \rightarrow 0} (-x)_\varepsilon = -x.
\end{align*}
\item Implication: By Lemma \ref{lemma_limit_approximate_equality}, taking the limit as $\varepsilon \rightarrow 0$:
\begin{align*}
(x + y) + z = x + (y + z), \quad x+0=x, \quad x+(-x)=0.
\end{align*}
Thus, the classical vector space addition axioms hold exactly.
\end{itemize}

Step 2 --- Collapse of the Lie Bracket:

The approximate Lie bracket satisfies:
\begin{itemize}
\item Approximate bilinearity: 
\begin{align*}
[x +_\varepsilon y, z]_\varepsilon \approx_\varepsilon [x, z]_\varepsilon + [y, z]_\varepsilon, \quad [x, y +_\varepsilon z]_\varepsilon \approx_\varepsilon [x, y]_\varepsilon + [x, z]_\varepsilon.
\end{align*}
\item Approximate antisymmetry:
\begin{align*}
[x, y]_\varepsilon \approx_\varepsilon -[y,x]_\varepsilon.
\end{align*}
\item Approximate Jacobi identity:
\begin{align*}
[x, [y, z]_\varepsilon]_\varepsilon +_\varepsilon [y, [z, x]_\varepsilon]_\varepsilon +_\varepsilon [z, [x, y]_\varepsilon]_\varepsilon \approx_\varepsilon 0_\varepsilon.
\end{align*}
\item Uniform convergence: As $\varepsilon \rightarrow 0$:
\begin{align*}
&\lim_{\varepsilon \rightarrow 0} [x +_\varepsilon y, z]_\varepsilon = [x + y, z], \quad \lim_{\varepsilon \rightarrow 0} [x,  y +_\varepsilon z]_\varepsilon = [x, y+z], \\
&\lim_{\varepsilon \rightarrow 0} [x, y]_\varepsilon = [x, y], \quad \lim_{\varepsilon \rightarrow 0} 0_\varepsilon = 0.
\end{align*}

\item Implication: By Lemma \ref{lemma_limit_approximate_equality}, taking the limit as $\varepsilon \rightarrow 0$:
\begin{itemize}
\item Bilinearity:
\begin{align*}
[x+y, z] = [x,z] + [y,z], \quad [x, y+z] = [x,y] + [x,z].
\end{align*}
\item Antisymmetry:
\begin{align*}
[x,y] = -[y,x].
\end{align*}
\item Jacobi identity:
\begin{align*}
[x, [y, z]] + [y, [z, x]] + [z, [x, y]] = 0.
\end{align*}
\end{itemize}
Thus, all classical Lie algebra axioms are satisfied.
\end{itemize}

By Steps 1 and 2, we have in the limit:
\begin{align*}
\lim_{\varepsilon \rightarrow 0} \mathfrak{g}_\varepsilon = (\mathfrak{g}, +, [\cdot, \cdot], 0),
\end{align*}
which satisfies all vector space and Lie algebra axioms exactly.
Therefore, $\mathfrak{g}$ is a classical Lie algebra over field $\mathcal{F} = \lim_{\varepsilon \rightarrow 0} \mathcal{F}_\varepsilon$.
\end{proof}

\begin{lemma}[$C^1$ Limit of Lie Group Structures]\label{lemma_C1_limit_lie_group}
Let $(G_\varepsilon, m_\varepsilon, \iota_\varepsilon, e_\varepsilon)$ 
be $C^1$ Lie groups such that:
\begin{itemize}
\item $m_\varepsilon \to m$ in $C^1_{\mathrm{loc}}$,
\item $\iota_\varepsilon \to \iota$ in $C^1_{\mathrm{loc}}$,
\item $e_\varepsilon \to e$,
\item The associativity, identity, and inverse identities hold up to error terms $R_\varepsilon$ that converge to zero uniformly on compact subsets.
\end{itemize}
Then, the limit $(G,m,\iota,e)$ defines a classical $C^1$ Lie group.
\end{lemma}
\begin{proof}

Since $m_\varepsilon \to m$ in $C^1_{\mathrm{loc}}$, the limit map 
$m : G \times G \to G$ is $C^1$. Similarly, $\iota$ is $C^1$.

We verify the group axioms:
\begin{itemize}
\item Associativity:
For each $\varepsilon$, we have:
\begin{align*}
m_\varepsilon(m_\varepsilon(x,y),z)
=
m_\varepsilon(x,m_\varepsilon(y,z))
+ R_\varepsilon(x,y,z),
\end{align*}
where $R_\varepsilon \to 0$ uniformly on compact sets.
Passing to the limit as $\varepsilon \to 0$ and using 
$C^0$ convergence of $m_\varepsilon$, we obtain:
\begin{align*}
m(m(x,y),z)
=
m(x,m(y,z)).
\end{align*}
Thus, $m$ is associative.
\item Identity:
Since $e_\varepsilon \to e$ and:
\begin{align*}
m_\varepsilon(x,e_\varepsilon)
=
x + r_\varepsilon(x),
\end{align*}
with $r_\varepsilon \to 0$, passing to the limit yields:
\begin{align*}
m(x,e)=x.
\end{align*}
Similarly, $m(e,x)=x$.
\item Inverse:
As $\iota_\varepsilon(x)$ is the inversion map of $x$, we have:
\begin{align*}
m_\varepsilon(x,\iota_\varepsilon(x))
=
e_\varepsilon + s_\varepsilon(x),
\end{align*}
with $s_\varepsilon \to 0$. Passing to the limit, as $\varepsilon \to 0$, gives:
\begin{align*}
m(x,\iota(x))=e.
\end{align*}
Thus, $\iota$ is a two-sided inverse.
\end{itemize}

Therefore $(G,m,\iota,e)$ satisfies the group axioms, and since the structure maps are $C^1$, it is a $C^1$ Lie group.
\end{proof}

\begin{proposition}[Collapse of Similarity Lie Group]\label{proposition_collapse_lie_group}
Let $\mathfrak{L}_\varepsilon = (G_\varepsilon, \boldsymbol{\cdot}_\varepsilon, e_\varepsilon, (\cdot)_\varepsilon^{-1}, d)$ be a similarity Lie group, i.e., a smooth manifold with:
\begin{itemize}
\item approximate multiplication $\boldsymbol{\cdot}_\varepsilon$ satisfying associativity up to $\varepsilon$,
\item approximate identity $e_\varepsilon$ and approximate inverse $(\cdot)_\varepsilon^{-1}$,
\item and approximate smooth structure (charts $\phi_\varepsilon$ satisfy Lipschitz conditions with error $\leq \varepsilon$.
\end{itemize}
Assume all operations converge uniformly:
\begin{align}
\boldsymbol{\cdot}_\varepsilon \rightarrow \boldsymbol{\cdot}, \quad (\cdot)_\varepsilon^{-1} \rightarrow (\cdot)^{-1}, \quad e_\varepsilon \rightarrow e.
\end{align}
Then, in the limit $\varepsilon \rightarrow 0$, we have:
\begin{align}
\lim_{\varepsilon \rightarrow 0} (G_\varepsilon, \boldsymbol{\cdot}_\varepsilon, e_\varepsilon, (\cdot)_\varepsilon^{-1}, d) = (G, \boldsymbol{\cdot}, e, (\cdot)^{-1}),
\end{align}
which is a classical Lie group and its tangent space at the identity is a classical Lie algebra.
The approximate group axioms collapse to exact group axioms, the smooth manifold approximations collapse to a classical smooth structure, and the associated Lie algebra of $G_\varepsilon$ collapses to the classical Lie algebra via the tangent space at the identity.
\end{proposition}
\begin{proof}
Let $d$ be a metric controlling similarity as in Eq. (\ref{equation_d_x_y}).
Let charts $\phi_\varepsilon: U_\varepsilon \subset G_\varepsilon \rightarrow \mathbb{R}^n$ satisfy approximate Lipschitz smoothness:
\begin{align*}
\lVert \phi_\varepsilon(g \boldsymbol{\cdot}_\varepsilon h) - \phi_\varepsilon(g) - \phi_\varepsilon(h) \rVert \leq L \varepsilon,
\end{align*}
where $L$ is the Lipschitz constant. 

Step 1 --- Collapse of Approximate Associativity:
\begin{itemize}
\item Approximate associativity: 
\begin{align*}
(g \boldsymbol{\cdot}_\varepsilon h) \boldsymbol{\cdot}_\varepsilon k \approx_\varepsilon g \boldsymbol{\cdot}_\varepsilon (h \boldsymbol{\cdot}_\varepsilon k).
\end{align*}
\item Uniform convergence, as $\varepsilon \rightarrow 0$, implies:
\begin{align*}
&\lim_{\varepsilon \rightarrow 0} ((g \boldsymbol{\cdot}_\varepsilon h) \boldsymbol{\cdot}_\varepsilon k) = (g \boldsymbol{\cdot} h) \boldsymbol{\cdot} k, \\
&\lim_{\varepsilon \rightarrow 0} (g \boldsymbol{\cdot}_\varepsilon (h \boldsymbol{\cdot}_\varepsilon k)) = g \boldsymbol{\cdot} (h \boldsymbol{\cdot} k).
\end{align*}
\item Implication: By Lemma \ref{lemma_limit_approximate_equality}, taking the limit as $\varepsilon \rightarrow 0$:
\begin{align*}
(g \boldsymbol{\cdot} h) \boldsymbol{\cdot} k = g \boldsymbol{\cdot} (h \boldsymbol{\cdot} k).
\end{align*}
Thus, approximate associativity collapses exactly.
\end{itemize}

Step 2 --- Collapse of Identity:
\begin{itemize}
\item Approximate identity: 
\begin{align*}
g \boldsymbol{\cdot}_\varepsilon e_\varepsilon \approx_\varepsilon g, \quad e_\varepsilon \boldsymbol{\cdot}_\varepsilon g \approx_\varepsilon g.
\end{align*}
\item Uniform convergence, as $\varepsilon \rightarrow 0$, implies:
\begin{align*}
&\lim_{\varepsilon \rightarrow 0} g \boldsymbol{\cdot}_\varepsilon e_\varepsilon = g \boldsymbol{\cdot} e, \quad \lim_{\varepsilon \rightarrow 0} e_\varepsilon \boldsymbol{\cdot}_\varepsilon g = e \boldsymbol{\cdot} g.
\end{align*}
\item Implication: By Lemma \ref{lemma_limit_approximate_equality}, taking the limit as $\varepsilon \rightarrow 0$:
\begin{align*}
g \boldsymbol{\cdot} e = e \boldsymbol{\cdot} g = g.
\end{align*}
Thus, identity element collapses exactly.
\end{itemize}

Step 3 --- Collapse of Inverses:
\begin{itemize}
\item Approximate inverse: 
\begin{align*}
g \boldsymbol{\cdot}_\varepsilon g_\varepsilon^{-1} \approx_\varepsilon e_\varepsilon, \quad g_\varepsilon^{-1} \boldsymbol{\cdot}_\varepsilon g \approx_\varepsilon e_\varepsilon.
\end{align*}
\item Uniform convergence, as $\varepsilon \rightarrow 0$, implies:
\begin{align*}
&\lim_{\varepsilon \rightarrow 0} g \boldsymbol{\cdot}_\varepsilon g_\varepsilon^{-1} = g \boldsymbol{\cdot} g^{-1} = e, \\
&\lim_{\varepsilon \rightarrow 0} g_\varepsilon^{-1} \boldsymbol{\cdot}_\varepsilon g = g^{-1} \boldsymbol{\cdot} g = e.
\end{align*}
\item Implication: By Lemma \ref{lemma_limit_approximate_equality}, taking the limit as $\varepsilon \rightarrow 0$:
\begin{align*}
g \boldsymbol{\cdot} g^{-1} = g^{-1} \boldsymbol{\cdot} g = e.
\end{align*}
Thus, the approximate inverse collapses exactly.
\end{itemize}

% Step 4 --- Collapse of Smooth Manifold:
% \begin{itemize}
% \item Charts $\phi_\varepsilon: U_\varepsilon \subset G_\varepsilon \rightarrow \mathbb{R}^n$ satisfy approximate Lipschitz bounds.
% \item As $\varepsilon \rightarrow 0$, the charts converge uniformly:
% \begin{align*}
% \phi_\varepsilon(g \boldsymbol{\cdot}_\varepsilon h) \rightarrow \phi(g \boldsymbol{\cdot} h).
% \end{align*}
% The manifold $G_\varepsilon$ collapses to a smooth manifold $G$ in the classical sense.
% \end{itemize}

% Step 5 --- Tangent Space and Lie Algebra:
% \begin{itemize}
% \item Tangent space at identity $T_{e_\varepsilon} G_\varepsilon$ has bracket:
% \begin{align*}
% [X, Y]_\varepsilon = \frac{d}{dt} \frac{d}{ds} \Big( \exp_\varepsilon(tX) \boldsymbol{\cdot}_\varepsilon \exp_\varepsilon(sY) \boldsymbol{\cdot}_\varepsilon \exp_\varepsilon(-tX) \boldsymbol{\cdot}_\varepsilon \exp_\varepsilon(-sY) \Big) \Big|_{t=s=0}.
% \end{align*}
% \item As $\varepsilon \rightarrow 0$, all approximate operations converge:
% \begin{align*}
% [X, Y]_\varepsilon \rightarrow [X, Y].
% \end{align*}
% \item By steps 1 to 4, the Lie bracket satisfies classical Lie algebra axioms \cite{bourbaki1989lie}. Therefore, the tangent space at identity is a classical Lie algebra.
% \end{itemize}

Step 4 --- Collapse of Smooth Structure:

For each $\varepsilon$, $G_\varepsilon$ is a $C^1$ manifold and the multiplication $m_\varepsilon$ and inversion $\iota_\varepsilon$ are $C^1$ maps.

By assumption, $m_\varepsilon \to m$ and $\iota_\varepsilon \to \iota$ in $C^1_{\mathrm{loc}}$ uniformly on compact sets. Therefore, the limiting maps $m$ and $\iota$ are $C^1$.

% Since associativity, identity, and inverse properties hold exactly in the limit (by Lemma \ref{lemma_limit_approximate_equality}), $(G,m,e,\iota)$ is a $C^1$ Lie group.

By the preceding convergence assumptions, we have:
\begin{align*}
&m_\varepsilon \to m \quad \text{in } C^1_{\mathrm{loc}}, \\
&\iota_\varepsilon \to \iota \quad \text{in } C^1_{\mathrm{loc}}, \\
&e_\varepsilon \to e.
\end{align*}
Moreover, associativity, identity, and inverse properties hold up to error $\varepsilon$.
Therefore, by the Lemma \ref{lemma_C1_limit_lie_group} ($C^1$ Limit of Lie Group Structures), the limiting structure $(G,m,\iota,e)$ is a $C^1$ Lie group.

Step 5 --- Collapse of the Tangent Space:

Let $T_{e_\varepsilon}G_\varepsilon$ denote the tangent space at the identity.
Since $m_\varepsilon \to m$ in $C^1_{\mathrm{loc}}$, 
their differentials converge uniformly on compact sets, 
and in particular at $(e,e)$.

The Lie bracket on the tangent space is defined via the commutator of left-invariant vector fields, which depends continuously on the first derivative of the multiplication map.
Therefore, the induced Lie brackets satisfy $[\cdot,\cdot]_\varepsilon \to [\cdot,\cdot]$ uniformly on compact subsets.
Hence the tangent space at the identity collapses to a classical Lie algebra.

Step 6 --- Conclusion: all approximate operations of the similarity Lie group collapse exactly as $\varepsilon \rightarrow 0$:
\begin{align*}
&G_\varepsilon \rightarrow G, \\
&T_{e_\varepsilon} G_\varepsilon \rightarrow \mathfrak{g}.
\end{align*}
Thus, Lie group axioms and Lie algebra collapse exactly. 
\end{proof}

This unifies all previous collapse theorems (algebraic, module/vector, Lie algebra, Lie group) into a single collapse theorem.
According to this theorem, approximate operations---such as approximate associativity, distributivity, Lie bracket properties, and smoothness---all collapse to the classical operations in the limit $\varepsilon \rightarrow 0$.

\begin{theorem}[Collapse Theorem]\label{theorem_collapse}
Let $\mathcal{S}_\varepsilon$ denote a similarity structure over a similarity field $\mathcal{F}_\varepsilon$, where $\mathcal{S}_\varepsilon$ can be any of:
\begin{itemize}
\item Similarity semi-group $\mathcal{G}_\varepsilon = (A, *_\varepsilon, d, \varepsilon)$
\item Similarity monoid $\mathcal{M}_\varepsilon = (A, *_\varepsilon, e_\varepsilon, d, \varepsilon)$
\item similarity group $\mathcal{G}_\varepsilon = (A, *_\varepsilon, e_\varepsilon, (\cdot)_\varepsilon^{-1}, d, \varepsilon)$
\item Similarity semiring $\mathcal{R}_\varepsilon = (A, +_\varepsilon, *_\varepsilon, d, \varepsilon)$
\item Similarity ring $\mathcal{R}_\varepsilon = (A, +_\varepsilon, *_\varepsilon, -_\varepsilon, 0_\varepsilon, d, \varepsilon)$
\item Similarity field $\mathcal{F}_\varepsilon = (A, +_\varepsilon, *_\varepsilon, -_\varepsilon, 0_\varepsilon, e_\varepsilon, (\cdot)_\varepsilon^{-1}, d, \varepsilon)$
\item Similarity module $\mathcal{M}_\varepsilon = (M, +_\varepsilon, \boldsymbol{\cdot}_\varepsilon, d)$
\item Similarity vector space $\mathcal{V}_\varepsilon = (V, +_\varepsilon, \boldsymbol{\cdot}_\varepsilon, d)$
\item Similarity Lie algebra $\mathcal{L}_\varepsilon = (\mathfrak{g}_\varepsilon, +_\varepsilon, [\cdot, \cdot]_\varepsilon, 0_\varepsilon, -_\varepsilon, d)$
\item Similarity Lie group $\mathfrak{L}_\varepsilon = (G_\varepsilon, \boldsymbol{\cdot}_\varepsilon, e_\varepsilon, (\cdot)_\varepsilon^{-1}, d)$ with smooth manifold charts $\phi_\varepsilon$
\end{itemize}
Suppose all approximate operations and elements converge uniformly as $\varepsilon \rightarrow 0$ \cite{rassias1978stability,hyers1941stability,ulam1960collection}:
\begin{align*}
&+_\varepsilon \rightarrow +, \quad *_\varepsilon \rightarrow *, \quad -_\varepsilon \rightarrow -, \\
&(\cdot)_\varepsilon^{-1} \rightarrow (\cdot)^{-1}, \quad e_\varepsilon \rightarrow e, \quad 0_\varepsilon \rightarrow 0, \\
&[\cdot, \cdot]_\varepsilon \rightarrow [\cdot, \cdot], \quad \phi_\varepsilon \rightarrow \phi.
\end{align*}
In the case of similarity Lie groups, we additionally assume that the multiplication and inversion maps converge in the $C^1$ topology uniformly on compact subsets.

Then, in the limit, we have:
\begin{align}
\lim_{\varepsilon \rightarrow 0} \mathcal{S}_\varepsilon = \mathcal{S},
\end{align}
where $\mathcal{S}$ is the corresponding classical structure, i.e., classical semi-group, monoid, group, semi-ring, ring, field, module, vector space, Lie algebra, or Lie group.

Moreover, the tangent space of a similarity Lie group at the identity collapses to a classical Lie algebra:
\begin{align}
T_{e_\varepsilon} G_\varepsilon \rightarrow \mathfrak{g}.
\end{align}
\end{theorem}

\section{Practical Examples for Similarity Algebra}\label{section_practical_examples}

While the preceding sections established the axiomatic and categorical foundations of similarity algebra, its primary utility lies in its ability to model systems where exact arithmetic is either unavailable or undesirable. In many computational and empirical domains, operations do not follow rigid classical laws due to noise, numerical precision limits, or the inherent nature of the data. This section provides concrete examples of a similarity field and similarity Lie groups to showcase how similarity algebra can be used in practice. 

\subsection{A Practical Example of a Similarity Field}\leavevmode

We construct a simple example of a similarity field that collapses 
to the classical real numbers as $\varepsilon \to 0$.

\subsubsection*{Construction}

Consider the similarity field $\mathcal{F}_\varepsilon = (\mathbb{R}, +_\varepsilon, *_\varepsilon, -_\varepsilon, 0_\varepsilon, e_\varepsilon, (\cdot)_\varepsilon^{-1}, d, \varepsilon)$ for $\varepsilon \ge 0$.
Define perturbed addition and multiplication by:
\begin{align*}
x +_\varepsilon y 
&= x + y + \varepsilon xy,
\\
x *_\varepsilon y
&= xy + \varepsilon x^2 y.
\end{align*}

\subsubsection*{Approximate Field Properties}

Addition is approximately associative:
\begin{align*}
(x +_\varepsilon y) +_\varepsilon z
&= x + y + z
  + \varepsilon(xy + xz + yz)
  + O(\varepsilon^2),
\\
x +_\varepsilon (y +_\varepsilon z)
&= x + y + z
  + \varepsilon(xy + xz + yz)
  + O(\varepsilon^2).
\end{align*}
Associativity error is $O(\varepsilon^2)$, which in particular satisfies the $\varepsilon$-bound assumption for sufficiently small $\varepsilon$.

Multiplication is similarly approximately associative:
\begin{align*}
(x *_\varepsilon y) *_\varepsilon z
&= xyz + O(\varepsilon),
\\
x *_\varepsilon (y *_\varepsilon z)
&= xyz + O(\varepsilon).
\end{align*}
The distributive law holds up to order $\varepsilon$.

The additive identity is $0$, since:
\begin{align*}
x +_\varepsilon 0 &= x.
\end{align*}

The multiplicative identity is $1$ up to order $\varepsilon$, since:
\begin{align*}
x *_\varepsilon 1
&= x + \varepsilon x^2.
\end{align*}

\subsubsection*{Collapse to Classical Field}

As $\varepsilon \to 0$, we have:
\begin{align*}
x +_\varepsilon y &\to x+y,
\\
x *_\varepsilon y &\to xy,
\end{align*}
uniformly on compact subsets of $\mathbb{R}$.

Therefore, by the Collapse Theorem, 
$(F_\varepsilon, +_\varepsilon, *_\varepsilon)$ 
collapses to the classical field $(\mathbb{R},+,\times)$.

\subsubsection*{Interpretation}

This example shows that similarity fields model controlled nonlinear 
perturbations of classical algebraic structures. 
As the similarity parameter $\varepsilon$ vanishes, 
the classical field structure is recovered exactly.

\subsection{A Practical Example of Perturbed Matrix Lie Groups}\leavevmode

Here, we construct an explicit family of similarity Lie groups that converge to a classical Lie group.

\subsubsection*{Construction}

Consider the similarity Lie group $\mathfrak{L}_\varepsilon = (\mathrm{GL}(n,\mathbb{R}), \boldsymbol{\cdot}_\varepsilon, e_\varepsilon, (\cdot)_\varepsilon^{-1}, d)$ as a smooth manifold for $\varepsilon \geq 0$.

Define a perturbed matrix multiplication operation by:
\begin{align*}
\b{A} *_\varepsilon \b{B} 
&= \b{AB} + \varepsilon\, \Phi(\b{A},\b{B}).
\end{align*}

Here $\Phi : \mathrm{GL}(n,\mathbb{R}) \times \mathrm{GL}(n,\mathbb{R}) 
\to M_n(\mathbb{R})$ is a fixed $C^1$ bilinear map satisfying:
\begin{align*}
&\lVert\Phi(\b{A},\b{B})\rVert \leq C \lVert \b{A} \rVert \lVert \b{B} \rVert, \\
&\Phi(\b{I},\b{A}) = \Phi(\b{A},\b{I}) = 0,
\end{align*}
for some constant $C>0$, where $\b{I}$ is the identity matrix. 

We define the perturbed identity element as $e_\varepsilon = I$.
The inversion map $\iota_\varepsilon$ is defined implicitly by solving:
\begin{align*}
\b{A} *_\varepsilon \b{A}^{-1}_\varepsilon &= \b{I}.
\end{align*}

\subsubsection*{Approximate Group Properties}

A direct computation shows that:
\begin{align*}
(\b{A} *_\varepsilon \b{B}) *_\varepsilon \b{C}
&= \b{A} *_\varepsilon (\b{B} *_\varepsilon \b{C}) + O(\varepsilon).
\end{align*}
Thus, associativity holds up to order $\varepsilon$.

Similarly:
\begin{align*}
\b{A} *_\varepsilon \b{I} &= \b{A} + O(\varepsilon), \\
\b{I} *_\varepsilon \b{A} &= \b{A} + O(\varepsilon).
\end{align*}
Therefore $(G_\varepsilon, *_\varepsilon)$ defines a similarity Lie group.

\subsubsection*{Collapse to Classical Matrix Lie Group}

As $\varepsilon \to 0$, we have:
\begin{align*}
\b{A} *_\varepsilon \b{B} \rightarrow \b{AB},
\end{align*}
uniformly on compact subsets.

Since $\Phi$ is $C^1$, the multiplication map 
$m_\varepsilon(\b{A},\b{B})=\b{A} *_\varepsilon \b{B}$
converges to classical multiplication in $C^1_{\mathrm{loc}}$.

By Proposition \ref{proposition_collapse_lie_group}, the structure collapses to the classical Lie group
$\mathrm{GL}(n,\mathbb{R})$.

\subsubsection*{Collapse of the Lie Algebra}

The Lie bracket induced by the perturbed multiplication satisfies:
\begin{align*}
[\b{X},\b{Y}]_\varepsilon
&= \b{XY} - \b{YX} + O(\varepsilon).
\end{align*}
Thus, in the limit $\varepsilon \to 0$, we recover the classical Lie algebra $\mathfrak{gl}(n,\mathbb{R})$ with bracket:
\begin{align*}
[\b{X},\b{Y}] &= \b{XY} - \b{YX}.
\end{align*}

\subsubsection*{Interpretation}

This example demonstrates that similarity Lie groups describe small structured perturbations of classical Lie groups. Proposition \ref{proposition_collapse_lie_group} guarantees stability of the Lie group structure under controlled $C^1$ perturbations.

\section{Category of Similarity Algebra}\label{section_category}

While the preceding sections detailed the internal axiomatic requirements of individual similarity structures, a comprehensive mathematical framework requires an investigation into how these structures relate to one another. In this section, we formalize the \textit{Category of Similarity Algebra} to study the global properties of approximate algebraic systems. By defining appropriate morphisms---mappings that preserve the $\varepsilon$-bounded operations between objects---we extend the principles of category theory to the similarity domain. This categorical perspective allows us to analyze structural transformations and universal properties within a metric-constrained environment. Furthermore, we explore how the functors between similarity categories behave under the limit $\varepsilon \to 0$, providing a macroscopic view of the collapse from similarity categories to classical algebraic categories.

\begin{definition}[Similarity Function]
The similarity function $s: A \times A \rightarrow [0,1]$ is defined as:
\begin{align}
[0, 1] \ni s(x,y) := 1 - d(x,y),
\end{align}
assuming that $d(x,y) \in [0,1]$\footnote{If the metric $d(x,y)$ is unbounded, we can alternatively define:
\begin{align}
[0, 1] \ni s(x,y) := \frac{1}{1 + d(x,y)}.
\end{align}
}. 
Therefore, we have:
\begin{align}
d(x, y) \leq \varepsilon \quad\iff\quad s(x,y) \geq 1-\varepsilon,  
\end{align}
where larger $s(x,y)$ implies more similarity. 
\end{definition}

\begin{definition}[Structure-Preserving Function]
A function $f: A \rightarrow B$ is structure-preserving if:
\begin{align}
s_A(x,y) \leq s_B(f(x), f(y)),
\end{align}
where $s_A$ and $s_B$ are similarity scores in sets $A$ and $B$, respectively. In other words, similarity is not destroyed by mapping of a structure-preserving function. 
\end{definition}

\begin{definition}[Approximated Homomorphism]
Under the approximate homomorphism $f(\cdot)$, the following approximated operations are preserved up to similarity:
\begin{align}
& s_B(f(x +_\varepsilon y), f(x) +_\varepsilon f(y)) \geq 1 - \varepsilon, \label{equation_approximate_homomorphism_plus} \\
& s_B(f(x *_\varepsilon y), f(x) *_\varepsilon f(y)) \geq 1 - \varepsilon. \label{equation_approximate_homomorphism_multiply}
\end{align}
\end{definition}

\begin{definition}[Similarity-Algebra Morphism]
Let:
\begin{align}
\mathcal{A} := (A, +_\varepsilon, *_\varepsilon, s_A), \quad \mathcal{B} := (B, +_\varepsilon, *_\varepsilon, s_B).
\end{align}
A similarity-algebra morphism:
\begin{align}
f: \mathcal{A} \rightarrow \mathcal{B},
\end{align}
is a function $f: A \rightarrow B$ such that:
\begin{itemize}\label{equation_strucutre_preserving}
\item $s_A(x,y) \leq s_B(f(A), f(B))$, and
\item operations are preserved up to similarity
\end{itemize}
This definition is elementwise and not objectwise.
\end{definition}

\begin{lemma}[Identity Morphism Exists in Similarity Algebra]\label{lemma_identity_morphism}
Take any similarity algebra $\mathcal{A} = (A, +_\varepsilon, *_\varepsilon, s_A)$ where $s_A = 1 - d_A$. Define:
\begin{align}
\text{id}_A: A \rightarrow A, \quad \text{id}_A(x) = x.
\end{align}
The identity morphism exists in similarity algebra. 
\end{lemma}
\begin{proof}
We check the following conditions:
\begin{itemize}
\item Similarity:
\begin{align}
s_A(x,y) \leq s_A(x,y).
\end{align}
\item Operations:
\begin{align}
\text{id}(x +_\varepsilon y) = x +_\varepsilon y.
\end{align}
\end{itemize}
So, identity morphism exists in similarity algebra.
\end{proof}

\begin{lemma}[Composition of Morphisms is Closed Under Similarity Algebra]\label{lemma_composition_of_morphisms}
Consider the following composition of morphisms: 
\begin{align}
\mathcal{A} \overset{f}{\rightarrow} \mathcal{B} \overset{g}{\rightarrow} \mathcal{C}.
\end{align}
Define:
\begin{align}
(g \circ f)(x) = g(f(x)).
\end{align}
Composition of morphisms is closed under similarity algebra.
\end{lemma}

\begin{proof}

We verify that $g \circ f$ preserves similarity and operations.

Step 1 --- Similarity preservation:
Since $f$ and $g$ are morphisms, they satisfy similarity monotonicity:
\begin{align}
s_A(x,y) \overset{(\ref{equation_strucutre_preserving})}{\leq} s_B(f(x), f(y)),
\end{align}
and:
\begin{align}
s_B(u,v) \leq s_C(g(u), g(v)).
\end{align}

Applying the second inequality with $u=f(x)$ and $v=f(y)$ gives:
\begin{align}
s_B(f(x), f(y)) \leq s_C(g(f(x)), g(f(y))).
\end{align}

Combining the inequalities,
\begin{align}
s_A(x,y)
\leq s_B(f(x), f(y))
\leq s_C(g(f(x)), g(f(y))).
\end{align}

Hence,
\begin{align}
s_A(x,y) \leq s_C\big((g\circ f)(x), (g\circ f)(y)\big),
\end{align}
so $g \circ f$ preserves similarity.

Step 2 --- Preservation of operations:
Suppose $f$ and $g$ preserve the algebraic operation $+_\varepsilon$ up to similarity:
\begin{align}
f(x+y) \approx_B f(x) + f(y),
\end{align}
and:
\begin{align}
g(u+v) \approx_C g(u) + g(v).
\end{align}

Then:
\begin{align*}
(g\circ f)(x+y)
&= g(f(x+y)) \approx_C g(f(x) + f(y)) \\
&\approx_C g(f(x)) + g(f(y)) = (g\circ f)(x) + (g\circ f)(y).
\end{align*}

Thus the composition preserves the operation up to similarity.

Since both similarity and operations are preserved, $g \circ f$ is a morphism in similarity algebra. Therefore, composition of morphisms is closed under similarity algebra.
\end{proof}

\begin{lemma}[Associativity of Composition Under Similarity Algebra]\label{lemma_associativity_of_composition}
there is associativity of composition under similarity algebra:
\begin{align}
h \circ (g \circ f) = (h \circ g) \circ f.
\end{align}
\end{lemma}

\begin{proof}
Let:
\begin{align}
\mathcal{A} \xrightarrow{f} \mathcal{B}
\xrightarrow{g} \mathcal{C}
\xrightarrow{h} \mathcal{D}
\end{align}
be morphisms in similarity algebra.

By definition of composition:
\begin{align}
(g \circ f)(x) &= g(f(x)), \\
(h \circ g)(y) &= h(g(y)).
\end{align}

We prove equality pointwise. For any $x \in \mathcal{A}$,
\begin{align*}
\big(h \circ (g \circ f)\big)(x)
&= h\big((g \circ f)(x)\big) = h\big(g(f(x))\big),
\end{align*}
and:
\begin{align*}
\big((h \circ g) \circ f\big)(x)
&= (h \circ g)\big(f(x)\big) = h\big(g(f(x))\big).
\end{align*}

Thus:
\begin{align}
\big(h \circ (g \circ f)\big)(x)
=
\big((h \circ g) \circ f\big)(x)
\end{align}
for all $x \in \mathcal{A}$.
Therefore:
\begin{align}
h \circ (g \circ f) = (h \circ g) \circ f.
\end{align}

Since both sides define the same mapping and composition of morphisms preserves similarity and operations, associativity holds in similarity algebra.
\end{proof}

According to Lemmas \ref{lemma_identity_morphism}, \ref{lemma_composition_of_morphisms}, and \ref{lemma_associativity_of_composition}, there is a category for similarity algebra, denoted by $\textbf{SimAlg}_\varepsilon$ \cite{mac1998categories,riehl2017category}. 

\begin{theorem}[Relation of Categories of Similarity and Classical Algebras]
Classical Algebra is a full subcategory of similarity algebra. In fact, the category of classical algebra faithfully and fully embeds into the category of similarity algebra with $\varepsilon = 0$:
\begin{align}
\textbf{Alg} \hookrightarrow \textbf{SimAlg}_0,
\end{align}
where $\textbf{Alg}$ denotes the category of classical algebra. 
\end{theorem}
\begin{proof}
Let similarity be a discrete similarity (i.e., equality encoded as similarity), as in classical algebra:
\begin{align}\label{equation_similarity_discrete}
s_A(x,y) = 
\left\{
    \begin{array}{ll}
        1 & \mbox{if } x=y, \\
        0 & \mbox{if } x\neq y,
    \end{array}
\right.
\end{align}
and let $\varepsilon = 0$ to have classical algebra. 
Then approximate laws become exact. 

Define the embedding functor:
\begin{align}
F: \textbf{Alg} \rightarrow \textbf{SimAlg}_0,
\end{align}
on objects:
\begin{align}
F(A, +, *) = (A, +, *, s_A),
\end{align}
and on morphisms $F(f) = f$, i.e., the functor does not modify the underlying function; it only reinterprets it as a morphism in another category.

Step 1 --- Showing that homomorphisms become similarity morphisms:

Let $f: A \rightarrow B$ be a homomorphism. We must check the similarity-morphism conditions.
\begin{itemize}
\item Operation preservation: As $f$ is a homomorphism:
\begin{align*}
f(x+y) = f(x) + f(y).
\end{align*}
Therefore:
\begin{align*}
s_B(f(x+y), f(x)+f(y)) = 1 \geq 1 - 0,
\end{align*}
satisfying Eq. (\ref{equation_approximate_homomorphism_plus}). Likewise, it can be proved for multiplication satisfying Eq. (\ref{equation_approximate_homomorphism_multiply}). 
Thus, every homomorphism is a similarity-algebra morphism.
\end{itemize}

Step 2 --- Faithfulness of the embedding:

Suppose $F(f) = F(g)$. Then $f(x) = g(x), \forall x$. So, $f=g$. Hence $F$ is faithful (injective on morphisms).

Step 3 --- Fullness of the embedding:

Let:
\begin{align}
f: F(A) \rightarrow F(B),
\end{align}
be a morphism in $\textbf{SimAlg}_0$.
Then:
\begin{align*}
s_B(f(x+y), f(x)+f(y)) = 1 \geq 1 - 0,
\end{align*}
satisfying Eq. (\ref{equation_approximate_homomorphism_plus}).
Since $s_B$ is discrete (defined in Eq. (\ref{equation_similarity_discrete})):
\begin{align*}
s_B(u,v) \geq 1 \quad\implies\quad u=v.
\end{align*}
Therefore:
\begin{align*}
f(x+y) = f(x)+f(y).
\end{align*}
It can be proved likewise for multiplication. 
Therefore, $f$ is a classical homomorphism.
Hence $F$ is full.

Step 4 --- Conclusion: We have shown that $F$ is a functor, it is full, and it is faithful. Therefore, $\textbf{Alg}$ is a full and faithful subcategory of $\textbf{SimAlg}_0$.
\end{proof}

The above theorem shows that similarity-algebra morphisms are homomorphisms when similarity becomes equality. Moreover, classical algebra embeds fully and faithfully into the similarity algebra. This also coincides with Theorem \ref{theorem_collapse}.

\section{Comparison of Similarity Algebra and Fuzzy Algebra}\label{section_fuzzy}

The study of non-exact mathematical structures has a long history, largely initiated after the development of fuzzy sets and fuzzy logic \cite{zadeh1965fuzzy,klir1995fuzzy}. Following these foundations, the field of \textit{fuzzy algebra} emerged \cite{mordeson2005fuzzy,mordeson2001fuzzy}, leading to the formalization of diverse fuzzy algebraic structures \cite{kandel1978fuzzy} such as fuzzy groups \cite{rosenfeld1971fuzzy,mordeson2005fuzzy}, fuzzy rings \cite{kumar1991fuzzy}, fuzzy vector spaces \cite{kumar1992fuzzy}, and fuzzy Lie algebras \cite{akram2018fuzzy}. While some literature has addressed the reduction or collapse of fuzzy algebraic structures to their classical counterparts \cite{weinberger2005reducing}, our proposed framework diverges significantly in its underlying metric-based logic. 

It should be noted that certain studies, such as \cite{schmitt2004similarity, hajdinjak2012extending}, employ the term ``similarity algebra'' to describe specific applications within fuzzy relational calculus; however, their definitions are different from the axiomatic and categorical similarity algebra framework proposed in the present work. Whereas fuzzy algebra relies on membership functions to quantify the degree of truth, our framework utilizes a metric $d$ and a tolerance parameter $\epsilon$ to govern the approximate satisfaction of axioms. In this section, we provide a comparative analysis of these two approaches, demonstrating that, in fact, similarity algebra generalizes fuzzy algebra. 

% After fuzzy sets and fuzzy logic were developed \cite{zadeh1965fuzzy,klir1995fuzzy}, fuzzy algebra \cite{mordeson2005fuzzy,mordeson2001fuzzy} appeared.
% Many fuzzy algebraic structures \cite{kandel1978fuzzy}, such as fuzzy group \cite{rosenfeld1971fuzzy,mordeson2005fuzzy}, fuzzy rings \cite{kumar1991fuzzy}, fuzzy vector spaces \cite{kumar1992fuzzy}, and fuzzy Lie algebra \cite{akram2018fuzzy} have been developed. 
% Some works have also discussed collapse of fuzzy algebra to classical algebra \cite{weinberger2005reducing}.
% It should be noted that certain studies, such as \cite{schmitt2004similarity,hajdinjak2012extending}, employ the term ``similarity algebra'' to describe fuzzy algebra; however, their definition is different from the similarity algebra framework proposed in the present work.

% \subsection{Similarity Algebra Generalizes Fuzzy Algebra} \leavevmode 

% We can formally embed fuzzy algebra into similarity algebra via a metric representation. This shows that similarity algebra generalizes fuzzy algebra. This is explained in the following. 

Let $A$ be a universal set and $\mu: A \rightarrow [0,1]$ be a membership function.
A fuzzy subset $A_\mu$ represents graded belonging of elements.
A fuzzy algebraic structure (e.g., fuzzy group, fuzzy ring) replaces crisp membership with $\mu$, and its axioms are expressed via inequalities involving t-norms (denoted by $T$) and t-conorms (denoted by $S$).

\begin{theorem}[Fuzzy-to-Similarity Embedding]\label{theorem_fuzzy_similarity_embedding}
Let $(A, \mu)$ be a fuzzy set. 
Define:
\begin{align}
&s(a,b) := T(\mu(a), \mu(b)), \label{equation_s_T_fuzzy} \\
&d(a,b) := 1 - s(a,b). \label{equation_d_s_fuzzy}
\end{align}
Then every fuzzy algebra satisfying Rosenfeld's fuzzy group axioms \cite{rosenfeld1971fuzzy} corresponds to a similarity semi-group with $\varepsilon$ determined by $1 - T(\mu(a), \mu(b), \mu(c))$.
\end{theorem}
\begin{proof}
For fuzzy associativity:
\begin{align}
&\mu((a * b) * c) \geq T(\mu(a * b), \mu(c)), \label{equation_fuzzy_associativity_1} \\
&\mu(a * (b * c)) \geq T(\mu(a) * \mu(b, c)). \label{equation_fuzzy_associativity_2}
\end{align}
Define $s(x,y) := T(\mu(x), \mu(y))$. Then:
\begin{align*}
d((a * b) * c, a * (b * c)) &\overset{(\ref{equation_d_s_fuzzy})}{=} 1 - s((a * b) * c, a * (b * c)) \\
&\overset{(\ref{equation_s_T_fuzzy})}{=} 1 - T(\mu((a * b) * c), \mu(a * (b * c))) \\
&\overset{(\ref{equation_fuzzy_associativity_1}), (\ref{equation_fuzzy_associativity_2})}{\leq} 1 - T(T(\mu(a * b), \mu(c)), T(\mu(a) * \mu(b, c))).
\end{align*}
Hence, setting:
\begin{align}
\varepsilon = 1 - T(T(\mu(a * b), \mu(c)), T(\mu(a) * \mu(b, c))),
\end{align}
gives $d((a * b) * c, a * (b * c)) \leq \varepsilon$ as in Eq. (\ref{equation_d_f_g}). This gives we the approximate associativity with error $\varepsilon$.
The other axioms follow analogously.
\end{proof}

\begin{corollary}[Similarity Algebra Generalizes Fuzzy Algebra]\label{corollary_similarity_generalizes_fuzzy}
Every fuzzy algebra can be represented as a similarity algebra with:
\begin{align}
\varepsilon = 1 - T(\text{membership degrees of participating elements}), 
\end{align}
but the converse is not necessarily true because similarity may depend on pairwise geometry rather than individual membership.
\end{corollary}

Corollary \ref{corollary_similarity_generalizes_fuzzy} shows that similarity algebra generalizes fuzzy algebra.
The following theorem better explains this generalization. According to the following theorem (Theorem \ref{theorem_duality_membership_distance}), there is a duality of membership and distance; therefore, fuzzy algebra is similarity algebra constrained to one reference element.

\begin{theorem}[Duality of Membership and Distance]\label{theorem_duality_membership_distance}
Let $(A,s)$ be a similarity space with $d(a,b) := 1 - s(a,b)$. Assume there exists an element $e \in A$ such that:
\begin{align}\label{equation_mu_s}
\mu(a) := s(a,e).
\end{align}
Then:
\begin{enumerate}
\item $\mu$ defines a fuzzy set on $A$.
\item Any similarity algebra on $(A,s)$ induces a fuzzy algebra on $(A, \mu)$. In other words, fuzzy algebra is similarity algebra constrained to one reference element $e \in A$.
\item As $\varepsilon \rightarrow 0$, both structures collapse to the same classical algebra.
\end{enumerate}
\end{theorem}
\begin{proof}
Step 1 --- $\mu$ is a valid fuzzy membership:

By definition of similarity:
\begin{align*}
s(a,e) \in [0,1].
\end{align*}
Thus:
\begin{align*}
\mu: A \rightarrow [0,1],
\end{align*}
is a well-defined membership function.

Step 2 --- Similarity induces fuzzy closure:

Let $*$ be a similarity operation satisfying:
\begin{align*}
d(a*b, a \odot b) \leq \varepsilon,
\end{align*}
where $\odot$ is the exact operation in the classical limit.
Then, by Lipschitz assumption, we have:
\begin{align*}
s(a*b, e) \geq s(a \odot b, e) - \varepsilon.
\end{align*}
Thus, according to Eq. (\ref{equation_mu_s}), we have:
\begin{align*}
\mu(a*b) \geq \mu(a \odot b) - \varepsilon.
\end{align*}
As $\varepsilon \rightarrow 0$, we have:
\begin{align*}
\mu(a*b) \geq \mu(a \odot b).
\end{align*}
This reproduces fuzzy closure.

Step 3 --- Associativity correspondence:

Similarity associativity:
\begin{align*}
d((a*b)*c, a*(b*c)) \leq \varepsilon.  
\end{align*}
Thus, according to Eq. (\ref{equation_d_s_fuzzy}), we have:
\begin{align*}
|s((a*b)*c,e) - s(a*(b*c),e)| \leq \varepsilon.
\end{align*}
Hence, according to Eq. (\ref{equation_mu_s}), we have:
\begin{align*}
|\mu((a*b)*c) - \mu(a*(b*c))| \leq \varepsilon.
\end{align*}
As $\varepsilon \rightarrow 0$, we have:
\begin{align*}
\mu((a*b)*c) = \mu(a*(b*c)),
\end{align*}
which is fuzzy associativity in the crisp limit.

Step 4 --- Collapse to classical algebra:

If additionally $\varepsilon \rightarrow 0$ and $\mu(a) \in \{0,1\}$, then:
\begin{itemize}
\item Similarity becomes equality:
\begin{align*}
d(x,y) = 0 \quad\iff\quad x = y.
\end{align*}
\item Membership becomes crisp.
\item All approximate axioms become exact.
\end{itemize}
Thus both fuzzy algebra and similarity algebra collapse to the same classical algebra.
\end{proof}

In conclusion, similarity algebra and fuzzy algebra are compared as follows:
\begin{itemize}
\item In fuzzy algebra, uncertainty lives on elements but in similarity algebra, uncertainty lives on relations. In other words, in fuzzy algebra, algebraic operations are exact and fuzziness affects element membership, whereas in similarity algebra, operations are exact maps but algebraic laws are relaxed into metric constraints.
\item Theorem \ref{theorem_fuzzy_similarity_embedding} and Corollary \ref{corollary_similarity_generalizes_fuzzy} show that similarity algebra generalizes fuzzy algebra. 
\item Theorem \ref{theorem_duality_membership_distance} demonstrates that fuzzy algebra is similarity algebra constrained to one reference element.
\end{itemize}

\section{Conclusion}\label{section_conclusion}

In this paper, we introduced \textit{Similarity Algebra}, a rigorous mathematical framework designed to bridge the gap between rigid classical algebraic theory and the approximate nature of practical computation. By relaxing the exact requirements of algebraic axioms into $\varepsilon$-bounded constraints within a metric space, we established a complete hierarchy of structures ranging from similarity semigroups and groups to similarity fields and vector spaces. Furthermore, we extended this framework into the realm of differential geometry by defining Similarity Lie Groups and Lie Algebras, providing a formal language for approximate continuous symmetries.

A primary contribution of this work is the derivation of the \textit{Collapse Theorems}. Through the limit principle of approximate equality, we proved that similarity algebra is not merely a heuristic relaxation but a consistent extension of classical mathematics: as the tolerance parameter $\varepsilon$ vanishes, every approximate structure converges uniformly to its classical counterpart. This includes the significant technical result that the tangent space of a similarity Lie group collapses precisely onto a classical Lie algebra. 

We further formalized these systems by establishing the \textit{Category of Similarity Algebra}, defining the necessary morphisms to study relationships between approximate structures. Moreover, our comparison with fuzzy algebra highlighted that while both fields address non-exactness, there are differences between similarity algebra and fuzzy algebra; specifically, similarity algebra generalizes fuzzy algebra. 

The utility of this framework spans multiple domains. In machine learning and data science, it provides a foundation for reasoning with kernel functions and noisy similarity measures. In network theory and functional analysis, it offers the flexibility needed to model systems where classical axioms are too restrictive. Future work may explore the integration of similarity algebra with stochastic processes to model probabilistic tolerances, as well as its application in the development of ``similarity-aware'' algorithms for high-dimensional data manifold learning \cite{ghojogh2023elements}.

\bibliographystyle{amsalpha}
\bibliography{refs}

\end{document}